\documentclass[12pt,reqno]{amsart}
\usepackage{amsmath, amssymb, latexsym, amsfonts, url, color, float}
\usepackage[vmargin=2.2cm,hmargin=2.6cm]{geometry} 

\usepackage{ifthen}
\newboolean{alefonts}
\setboolean{alefonts}{true}

\ifthenelse{\boolean{alefonts}}{

\usepackage[full]{textcomp} 
\usepackage{newtxtext} 
\usepackage{cabin} 
\usepackage{zlmtt}
\usepackage[bigdelims,vvarbb]{newtxmath} 
\usepackage[cal=boondoxo]{mathalfa} 

}{}

\newtheoremstyle{sltheorems}
{10pt}
{6pt}
{\slshape}
{}
{\bfseries}
{.}
{.5em}
{\thmname{#1}\thmnumber{ #2}\thmnote{ (#3)}}

\theoremstyle{sltheorems} 
\newtheorem{Thm}{Theorem}

\newtheorem{Lem}{Lemma}
\newtheorem{Def}{Definition}
\newtheorem{Prop}{Proposition}

\newtheorem{Cor}{Corollary}

\usepackage[pdftex,colorlinks=true,linkcolor=black,urlcolor= blue,draft=false]{hyperref}
\def\urltwo#1#2{\mbox{\href{#1}{\texttt{#2}}}}

\usepackage{amsthm}
\newtheoremstyle{remark}
{10pt}
{6pt}
{\rm} 
{}
{\bfseries}
{.}
{.5em}
{\thmname{#1}\thmnumber{ #2}\thmnote{ (#3)}}
\theoremstyle{remark} 
\newtheorem{Rem}{Remark}
\newtheorem{Ex}{Example}

\makeatletter
\let\@@pmod\pmod
\DeclareRobustCommand{\pmod}{\@ifstar\@pmods\@@pmod}
\def\@pmods#1{\mkern4mu({\operator@font mod}\mkern 6mu#1)}
\makeatother

\usepackage{microtype} 

\usepackage{caption} 
\makeatletter
\patchcmd{\env@cases}{1.2}{1}{}{}
\makeatother

\newcommand{\numSconstantdecimals}{$130\,000$ }

\usepackage{paralist}

  \pltopsep=\medskipamount
\newcommand{\custspace}{\hspace{.3em}}  

\usepackage{graphicx}

\begin{document}
\title[Euler constants from primes in arithmetic progression]{Euler constants from primes\\ in arithmetic progression}
\author{Alessandro Languasco and Pieter Moree}

\date{}

\subjclass{01A99, 11N37, 11Y60}

\begin{abstract}
\noindent Many Dirichlet series of number theoretic interest can be written as a product of 
generating series $\zeta_{\,d,a}(s)=\prod\limits_{p\equiv a\pmod*{d}}(1-p^{-s})^{-1}$, 
with $p$ ranging over all the primes in the primitive residue class modulo $a\pmod*{d}$, and a function $H(s)$ well-behaved  around $s=1$. In such a case the corresponding Euler constant can be 
expressed in terms of the Euler constants $\gamma(d,a)$ of the series $\zeta_{\,d,a}(s)$ 
involved and the (numerically more harmless) term $H'(1)/H(1)$. Here we systematically 
study $\gamma(d,a)$, their numerical evaluation 
and discuss some examples.
\end{abstract}

\maketitle

\vspace{-0.6cm}
\section{Introduction}
Given any complex number $s$ having 
real part $>1$, the Riemann zeta function $\zeta(s)$ satisfies
\begin{equation*}
 \zeta(s) = \sum_{n=1}^\infty \frac{1}{n^s}=\prod_p\frac{1}{1-p^{-s}},
\end{equation*}
where here and in the sequel the letter $p$ is exclusively used to denote
prime numbers.
The Laurent expansion of $\zeta(s)$ around $s=1$ be written in the form
\begin{equation}
\label{taylor}    
\zeta(s)=\frac{1}{s-1}+\sum_{j=0}^{\infty}\frac{(-1)^j}{j!}\gamma_j\,(s-1)^j.
\end{equation}
It has been independently observed by various authors (for example Briggs and Chowla \cite{BC})  that
for every $j \ge 0$ one has
\begin{equation}
\label{stieltjes}
\gamma_j=\lim_{N\rightarrow \infty}
\Bigl(\sum_{k=1}^N \frac{(\log k)^j}{k}-\frac{(\log N)^{j+1}}{j+1}\Bigr).
\end{equation}
The most famous of these so-called \emph{Stieltjes} constants (see, e.g., Coffey \cite{Coffey}) is
$$\gamma=\gamma_0=0.57721 56649 01532 86060 65120 90082 40243 10421 59335 93992 35988 05767\dotsc,$$
the \emph{Euler constant}\footnote{Sometimes called \emph{Euler-Mascheroni constant}. 
Lagarias \cite{Lagarias} in his very readable survey argues that the name Euler constant is more appropriate.}.
Diamond \cite{Diamond}, Kurokawa and  Wakayama \cite{KW} and various 
other authors considered $p$-adic and $q$-analogues of 
Euler constants. Many further generalizations of Euler constants are conceivable, cf.\,Knopfmacher \cite{Knopf} 
and Hashimoto \cite{Hashimoto}.
One of these arises on letting the integers $k$ in \eqref{stieltjes} run 
through an arithmetic progression \cite{Dilcher, Lehmer, Shirasaka}. Here we consider another generalization 
obtained by restricting the primes in the Euler product for $\zeta(s)$ to be in a prescribed 
primitive arithmetic progression $a\pmod*{d}$, that is we will consider
\begin{equation}
\label{Fda-def}
\zeta_{\,d,a}(s)=\prod_{p\equiv a\pmod*{d}}\frac{1}{1-p^{-s}},
\end{equation}
with $d$ and $a$ coprime.
In 1909, Landau \cite{La09a} 
(see also \cite[\S 176--183, Vol.~II]{Lehrbuch})
in the course of settling a question 
of D.N.~Lehmer,  
considered $\zeta_{\,d,a}(s)$.
Moree \cite{bias} 
studied $\zeta_{\,d,a}(s)$ 
for $d=3,4$, $\zeta_{\,12,1}(s)$ was considered by Fouvry
et al.\,\cite{FLW} and earlier by 
Serre \cite[pp.~185--187]{SerreChebotarev}.
In case $d$ is an odd prime, the functions $\zeta_{\,d,1}(s)$ were implicitly 
studied in \cite{FLM,SW}. Undoubtedly further examples can be found in the literature.

It follows from \eqref{taylor} that
\begin{equation*}
\gamma=\lim_{s\rightarrow 1^+}\Bigl(
\frac{\zeta'}{\zeta}(s) 
+\frac{1}{s-1}\Bigr),
\end{equation*}
where here and in the rest of the article the shorthand $F'/F(s)$ stands for $F'(s)/F(s)$.
Likewise, given a Dirichlet series $F(s)$ one can wonder whether 
\begin{equation*}
\gamma_F=\lim_{s\rightarrow 1^+}\Bigl(
\frac{F'}{F}(s) 
+\frac{\alpha}{s-1}\Bigr),
\end{equation*}
exists for an appropriate choice of $\alpha$. If so, we can regard it as the
Euler constant associated to $F.$ We can now define $\gamma(d,a)$ as
\begin{equation*}
\label{gammada-def} 
\gamma(d,a)=\lim_{s\rightarrow 1^+}\Bigl(
\frac{\zeta_{\,d,a}'}{\zeta_{\,d,a}}(s) 
+\frac{1}{\varphi(d)(s-1)}\Bigr),
\end{equation*}
where $\varphi$ denotes Euler's totient function.
The problem we address in this article is how to compute these constants with a high
numerical precision. For example, the following formula 
\begin{equation}
\label{nicelimit}
\gamma(d,a)=\lim_{x\rightarrow \infty}\Bigl(\frac{\log x}{\varphi(d)} - \sum_{\substack{p\le x\\ p\equiv a\pmod*{d}}}\frac{\log p}{p-1}\Bigr).
\end{equation}
valid for coprime $a$ and $d$ is very compact, but cannot be used for a rigorous estimate. It is easily derived
from the expression
\begin{equation*}
\gamma(d,a)=\lim_{x\rightarrow \infty}
\Bigl(\frac{\log x}{\varphi(d)} - 
\sum_{\substack{n\le x\\ n\equiv a\pmod*{d}}}\frac{\Lambda(n)}{n}\Bigr),
\end{equation*}
where 
$\Lambda$ denotes the von Mangoldt function 
(the latter two limit results
are special cases of Moree \cite[Thm.\,3]{LvR}). 

Our main result (proven in 
Section \ref{sec:gammaproperties}) gives a less elegant formula, but it can 
be used to obtain $\gamma(d,a)$ with high numerical precision (at least when $d$ is not too large). 
\begin{Thm}
\label{mainabelian}
Let $a$ and $d\ge 2$ be coprime integers. We have
\begin{equation}
    \label{compact}
\gamma(d,a)=\gamma_1(d,a)-\sum_{p\equiv a\pmod*{d}}\frac{\log p}{p(p-1)}
+
\sum_{n\equiv a\pmod*{d}}\frac{(1+\mu(n))\Lambda(n)}{n},
\end{equation}
where $\mu$ is the M\"obius function and
\begin{equation}
\label{gamma1-def}
\gamma_1(d,a)
=
\frac{1}{\varphi(d)}\Bigl(\gamma+\sum_{p\mid d}\frac{\log p}{p-1}+
\sum_{\chi\ne \chi_0}{\overline{\chi}}(a)
\frac{L'}{L}(1,\chi)\Bigr),
\end{equation}
and $L(s,\chi)$ denotes a Dirichlet $L$-function.
\end{Thm}

\begin{Cor}
\label{Corollary-abelian}
Suppose that
\begin{equation}
\label{abelianformat}    
F(s)=H(s)\prod_{\substack{a=1\\ (d,a)=1}}^{d-1}\zeta_{\,d,a}(s)^{e_a},
\end{equation}
with the $e_a$ rational numbers and
$H(s)$ a Dirichlet series that it is holomorphic, non-zero
and uniformly bounded
in $\Re(s)>1-\varepsilon$ for some $\varepsilon>0$, then
\begin{equation}
\label{generalident}    
\gamma_F=\frac{H'}{H}(1)+\sum_{\substack{a=1\\ (d,a)=1}}^{d-1} e_a\,\gamma(d,a).
\end{equation}
\end{Cor}
The number of decimals we are able to determine of
$H'/H(1)$ is quite variable. For example, if
$H$ is a product
of factors $\zeta_{d,a}(f_as)^{g_a}$, for some $f_a>1$, $g_a$ rational, then $\gamma_F$ will also depend on 
$g_a\, \zeta'_{d,a}/\zeta_{d,a}(f_a)$, and this is a prime sum that we are able to evaluate
carefully if $d$ is not too large, see Remark \ref{Remark-sums}. 

As the literature abounds with Dirichlet series satisfying the format \eqref{abelianformat}, this shows the relevance
of being able to compute $\gamma(d,a).$ 
To prove Corollary \ref{Corollary-abelian} it is enough to compute the logarithmic derivative 
of the quantities in \eqref{abelianformat}.

It is easy to see that the first
sum in \eqref{compact} converges.
Since $(1+\mu(n))\Lambda(n)$ is non-zero only if $n=p^k$ with $p$ a prime and $k\ge 2$, also the second sum in \eqref{compact} converges.
Indeed, both sums are bounded by
$\sum_p \frac{\log p}{p(p-1)}<0.7554.$ 
The second sum in \eqref{compact}
can be rewritten so that it only involves primes, 
see Section \ref{sec:second-sum}.

Letting $a$ and $d\ge 2$ be coprime integers, the limit for $s\to 1^+$ of \eqref{Fda-def} diverges 
and hence the quantity
\[
P(x;d,a) = 
  \prod_{\substack{p \leq x \\ p \equiv a \bmod d}}
  \Bigl( 1 - \frac1p \Bigr)
\]
tends to zero as $x$ tends to infinity. The asymptotic behavior of $P(x;d,a)$ was 
first studied in 1909 by Landau \cite[\S 7 and \S 110, Vol.~I]{Lehrbuch}\footnote{In fact, he 
worked on the equivalent problem of estimating
$\sum\limits_{\substack{p \leq x \\ p \equiv a \bmod d}}\frac{1}{p}$.}.
One has
\begin{equation}
\label{MertensAP}
P(x;d,a) =   C(d, a)(\log x)^{-\frac{1}{\varphi(d)}}
  +
  O( (\log x)^{-1-\frac{1}{\varphi(d)}})
\end{equation}
as $x \to +\infty$, with $C(d,a)$ a real and positive constant, 
usually called \emph{Mertens' constant for primes in arithmetic progressions}.
In 1974, Williams \cite{Williams74} 
gave an explicit expression for it.
In some special cases it was determined by Uchiyama \cite{Uchiyama71}
($d = 4$), Grosswald \cite{Grosswald87} ($d \in \{4$, $6$, $8\}$) and Williams \cite{Williams74} ($d =24$).

In 2007, Languasco and Zaccagnini \cite{lanzac07} proved 
an uniform (in the $d$-aspect) version of \eqref{MertensAP}
and gave the following elementary expression for $C(d,a)$:
\[
  C(d, a)^{\varphi(d)}
  =
  e^{-\gamma}
  \prod_p
    \Bigl( 1 - \frac 1p \Bigr)^{\alpha(p; d, a)},
\]
where $\alpha(p; d, a) = \varphi(d) - 1$ if $p \equiv a \bmod d$ and
$\alpha(p; d, a) = - 1$ otherwise.
In particular, such an expression gave the possibility 
of performing high precision computations on $C(d, a)$
and other related quantities, see \cite{lanzac09,lanzac10b,lanzac10},
at least when $d$ is not too large.
In \cite{lanzac08}, they 
established Bombieri--Vinogradov and Barban--Davenport--Halberstam type results
for $P(x;q,a)$.
 
We point out that often also
$\zeta_{\,d,a}(2)$ is relevant. For 
the high numerical precision evaluation of this Euler product, see, e.g., 
Ettahri--Ramar\'e--Surel \cite{ERS} or Ramar\'e \cite{Accurate,walk}.

The paper is organized as follows:
Section  \ref{sec:mainabelian-proof} 
gives a proof of Theorem \ref{mainabelian}, while
in Section \ref{sec:second-sum} we will give an alternative expression for 
the second sum in \eqref{compact} involving just primes, rather than prime powers. 
Euler constants of multiplicative sets 
are  discussed in Section \ref{sec:EulerS} and how to
express them in terms of $\gamma(d,a)$ in case the set is multiplicative abelian. It 
was in the framework of multiplicative 
sets that the authors got the idea of 
considering $\gamma(d,a)$.
Section \ref{sec:properties-specialcases} shows some 
useful properties of $\gamma(d,a)$ and its shape in some important special cases.
In particular, here the (historically important) cases $d=3$ and $d=4$ are worked out in detail. 
In Section \ref{sec:algorithm} the computation of \texorpdfstring{$\gamma(d,a)$}{gda} 
in the general case is considered. 
Section \ref{sec:Dedekind} is about the Euler constants of number fields 
(which can be expressed in terms of $\gamma(d,a)$ in case of an abelian number field). 
In Section \ref{sec:FouvrySerre} 
we will point out that the seemingly unrelated results of 
Fouvry et al.~and Serre for $d=12$ mentioned 
above, are actually connected.
Table \ref{tab:24} gives the numerical values of $\gamma(d,a)$ for 
$d=1,2,3,4,5,7,8,9$ and $12$ and $1\le a\le d-1$ coprime to $d$.
Finally, in Appendix \ref{sec:appendix} we give formulae relating
$\gamma(d,a)$ with some prime sums for $d=5,7,8,9$ and $12$ and
$1\le a\le d-1$ coprime to $d$. These rest on 
the evaluation \eqref{compacttwo} of the second sum in \eqref{compact}.

\section{Proofs of the results announced in the introduction}
Theorem \ref{mainabelian} is an easy result, but at the heart of our paper.
\subsection{Proof of Theorem \ref{mainabelian}}
\label{sec:mainabelian-proof}
Recalling that for $\Re(s)>1$ one has
\[
\log L(s,\chi) = \sum_{n=2}^\infty \frac{\chi(n) \Lambda(n)}{(\log n) n^s},
\]
by character orthogonality we immediately obtain
\[
\frac{1}{\varphi(d)}\sum_{\chi}{\overline{\chi}}(a)\log L(s,\chi)
=
\sum_{n\equiv a \pmod*{d}} \frac{\Lambda(n)}{(\log n) n^s}
=
\sum_{p\equiv a \pmod*{d}} \frac{1}{p^s}
+\sum_{\substack{{n\equiv a \pmod*{d}}\\ n\, \textrm{composite}}} \frac{\Lambda(n)}{(\log n) n^s}.
\]
Moreover, we have 
\[
\log \zeta_{\,d,a}(s) =
\sum_{p\equiv a \pmod*{d}} \sum_{k \ge 1}\frac{1}{kp^{ks}}
=
\sum_{p\equiv a \pmod*{d}} \frac{1}{p^s}
+
\sum_{p\equiv a \pmod*{d}} \sum_{k \ge 2}\frac{1}{kp^{ks}}
\]
and, by subtracting the last two formulae we can write
\[
\log \zeta_{\,d,a}(s) 
-
\frac{1}{\varphi(d)}\sum_{\chi}{\overline{\chi}}(a)\log L(s,\chi)
=
\sum_{p\equiv a \pmod*{d}} \sum_{k \ge 2}\frac{1}{kp^{ks}}
-
\sum_{\substack{{n\equiv a \pmod*{d}}\\ n\, \textrm{composite}}} \frac{\Lambda(n)}{(\log n) n^s}.
\]
Recalling that $L(s,\chi_0) = \zeta(s)\prod_{p\mid d}(1-p^{-s})$, 
differentiating and letting $s\to 1^+$ both sides, 
we finally obtain
\begin{align*}
\gamma(d,a)
-
\gamma_1(d,a)
&=
-
\sum_{p\equiv a \pmod*{d}} \log p \sum_{k \ge 2}\frac{1}{p^{k}}
+
\sum_{\substack{{n\equiv a \pmod*{d}}\\ n\, \textrm{composite}}} \frac{\Lambda(n)}{n}
\\&
= 
- \sum_{p\equiv a \pmod*{d}} \frac{\log p}{p(p-1)}
+
\sum_{n\equiv a\pmod*{d}}\frac{(1+\mu(n))\Lambda(n)}{n},
\end{align*}
in which we used the sum of a geometric series of ratio $1/p$,
the fact that $\Lambda(n)\ne 0$ if and only if $n$ is a prime power,
$\mu(p)=-1$ and $\mu(p^\ell)=0$ for $\ell\ge 2$.
\qed

\subsection{Evaluation of the second sum in \texorpdfstring{\eqref{compact}}{eqrefcompact}}
\label{sec:second-sum}
We will now evaluate the second sum in \eqref{compact} more explicitly, 
identifying the contribution of the prime powers $p^k$ with $k\ge 2$ according
to the primitive arithmetic progression 
$b\pmod*{d}$ the prime $p$ is in. 
In this very elementary analysis 
the \emph{multiplicative order} $\nu_d(b)$ 
of $b$ modulo $d$, the smallest integer $k\ge 1$ such that $b^k\equiv 1\pmod*{d}$, will play an important role. Note that $\nu_d(b)=\#\langle b\rangle$, the order of the subgroup modulo $d$ generated by $b$.

The equation $X^k\equiv a\pmod*{d}$ with $X$ in prescribed primitive
residue class  $b\pmod*{d}$ has a solution 
$k\ge 1$ if and only if $a\in \langle b\rangle$. In case $a\in \langle b\rangle$, we let $k_0$ be the smallest
integer such that 
$b^{k_0}\equiv a\pmod*{d}$ and $k_0\ge 2$. In 
case $a=1$ we obtain 
$k_0=\max\{2,\nu_d(b)\}$, if $a=b$ we have
$k_0=\nu_d(a)+1$.
We write $e_b=k_0$. 
If $b^k\equiv a\pmod*{d}$ has a solution 
with $k\ge 2$, then its solutions $k\ge 2$ are precisely given by the arithmetic progression $e_b,e_b+\nu_d(b),e_b+2\nu_d(b),\ldots$
Notice that the sum $\sum_j p^{-j}$, with $j$ running through this arithmetic progression, equals $p^{\nu_d(b)-e_b}/(p^{\nu_d(b)}-1)$, which is 
$O(p^{-2})$.
Sorting the prime powers $p^k$ according to the congruence class of $p$ modulo $d$, we obtain
\begin{equation}
    \label{compacttwo}
\sum_{n\equiv a\pmod*{d}}\frac{(1+\mu(n))\Lambda(n)}{n}=
\sum^{d-1}_{\substack{b=1 \\(d,b)=1\\ a\in \langle b\rangle}}
\sum_{p\equiv b\pmod*{d}}
\frac{p^{\nu_d(b)-e_b}\log p}{p^{\nu_d(b)}-1}.
\end{equation}
Note that if $a\in \langle b\rangle$, then $\langle a\rangle$ is a subgroup of 
$\langle b\rangle$
and hence $\nu_d(a)\mid \nu_d(b)$ by Lagrange's theorem. When $d$ is prime, then
$(\mathbb Z/d\mathbb Z)^*$ is cyclic and this is if and only if. Hence if $d$ is prime, 
the condition $a\in \langle b\rangle$ in the right hand side in \eqref{compacttwo} 
can be replaced by $\nu_d(a)\mid \nu_d(b)$.

\section{Euler constants of multiplicative sets}
\label{sec:EulerS}
A set $S$ of natural numbers is said to be \emph{multiplicative} if for every pair 
$m$ and $n$ of coprime integers in $S$,  $mn$ is in $S$. 
Let $\chi_S$ denote the 
\emph{characteristic function} of $S$. 
It is $1$ if $n$ is in $S$ and $0$ otherwise. Note that the set $S$ is multiplicative if and only if $\chi_S$ is a multiplicative function.
By $\pi_S(x)$ and $S(x)$ we denote the number of primes, respectively integers in $S$ not exceeding $x.$
The following result is a special case of a theorem 
of Wirsing \cite{Wirsing}, with a reformulation following Finch et al.\,\cite[p.\,2732]{FMS}. 
\begin{Thm}
\label{een}
Let $S$ be a multiplicative set satisfying $\pi_S(x)\sim \delta x/\log x$, as $x\to\infty$,  for some $0<\delta<1$.  
Then $$S(x)\sim c_0(S) x \log^{\delta-1}x,$$
where 
$$c_0(S)=\frac{1}{\Gamma(\delta)}\lim_{P\rightarrow \infty}\prod_{p<P}
\Bigl(1+\frac{\chi_S(p)}{p}+\frac{\chi_S(p^2)}{p^2}
+\dotsm \Bigr)\Bigl(1-\frac{1}{p}\Bigr)^{\delta},$$
converges and hence is positive.
\end{Thm}

For $\Re(s)>1$ we put 
$$L_S(s) =\sum_{n\in S}\frac1{n^s}.$$
If the limit
\begin{equation*}
\gamma_S =\lim_{s\rightarrow 1^+}\Bigl(
\frac{L'_S(s)}{L_S(s)} 
+\frac{\alpha}{s-1}\Bigr)
\end{equation*}
exists for some $\alpha>0$, we say that the set $S$ admits an \emph{Euler constant} $\gamma_S$. In case
$S=\mathbb N,$ we have of course 
$L_S(s)=\zeta(s)$, $\alpha=1$ and $\gamma_S=\gamma$. 

As the following result shows, the Euler constant $\gamma_S$ determines 
the second order behavior of $S(x).$ 
\begin{Thm}
\label{thm:multiplicativeset}
Let $S$ be a multiplicative set.
If there are $\rho>0$ and $0<\delta<1$ such that
\begin{equation}
\label{primecondition}
\pi_S(x)=
\delta\pi(x)+O_S(x\log^{-2-\rho}x),
\end{equation}
then $\gamma_S\in\mathbb R$ exists and asymptotically 
\begin{equation*}
S(x)=\sum_{n\le x,\,n\in S}1=\frac{c_0(S)\,x}{\log^{1-\delta}x}\Bigl(1+\frac{(1-\gamma_S)(1-\delta)}{\log x}\bigl(1+o_S(1)\bigr)\Bigr).
\end{equation*}
In case  the prime numbers belonging to $S$ are, with finitely many exceptions, precisely
those in a finite union of arithmetic progressions, 
we have, for arbitrary $j\ge 1$,
\begin{equation}
\label{starrie1}
S(x)=\frac{c_0(S)\,x}{\log^{1-\delta}x}\Bigl(1+\frac{c_1(S)}{\log x}+\frac{c_2(S)}{\log^2 x}+\dots+
\frac{c_j(S)}{\log^j x}+O_{j,S}\Bigl(\frac{1}{\log^{j+1}x}\Bigr)\Bigr),
\end{equation}
with  $c_1(S)=(1-\gamma_S)(1-\delta)$ and
$c_2(S),\dotsc,c_j(S)$ constants.
\end{Thm}
There are many results in this spirit in the literature. For (\ref{primecondition}), see Moree \mbox{\cite[Theorem 4]{Mpreprint}}; for 
\eqref{starrie1} the reader can consult, for example, Serre \cite[Th\'eor\`eme~2.8]{serre}.

The multiplicative set $S_{d,a}$ associated to $\zeta_{\,d,a}(s)$ consists of the natural numbers having
only prime factors $p\equiv a\pmod*{d}$.
By the prime number theorem for arithmetic progressions \eqref{primecondition} 
is satisfied with $\delta=1/\varphi(d)$.
 Theorem \ref{thm:multiplicativeset} then yields
\begin{equation*}
S_{d,a}(x)=
\frac{c_0(S_{d,a})\,x}{\log^{1-\frac1{\varphi(d)}}x}
\Bigl(1+\frac{(1-\gamma(d,a))(1-\frac1{\varphi(d)})}{\log x}+\dots+
\frac{c_j(S)}{\log^j x}+O_{j,d}\Bigl(\frac{1}{\log^{j+1}x}\Bigr)\Bigr),
\end{equation*}
where the implicit constant depends at most on $j$ and $d$.
\begin{Ex} Using Theorem \ref{een} it can be shown that
$c_0(S_{4,1})=0.32712936\ldots$ 
and 
$c_0(S_{4,3})=0.48651988\ldots$,
and hence $S_{4,3}(x)\ge S_{4,1}(x)$
for every $x$ sufficiently large. Moree \cite{bias} showed that $S_{4,3}(x)\ge S_{4,1}(x)$ for \emph{every} $x$.
Further, 
$c_0(S_{3,1})=0.30121655\ldots$ 
and 
$c_0(S_{3,2})=0.70449843\ldots$ 
and so $S_{3,2}(x)\ge S_{3,1}(x)$
for every $x$ sufficiently large. Moree \cite{bias} 
also showed that this inequality holds for \emph{every} $x$.
In contrast, denoting by $\pi(x;d,a)$ the number of primes $p\le x$, $p\equiv a \pmod*{d}$,
it is well-known that the quantities
$\pi(x;4,3)-\pi(x;4,1)$ and $\pi(x;3,2)-\pi(x;3,1)$, although predominantly 
positive (Chebyshev's bias, see 
Granville and Martin \cite{GM}), change sign infinitely often. For related material 
see Moree \cite[Sect.\,6]{LvR}.
\end{Ex}
Theorem \ref{thm:multiplicativeset} allows us to determine
whether the \emph{Landau approximation}, $c_0(S)\,x\log^{\delta-1}x$, or
the \emph{Ramanujan approximation}, $c_0(S)\,\int_2^{x}t\log^{\delta-1}t$, asymptotically better
approximates $S(x)$.
\begin{Cor}
If $\gamma_S<\frac12$, then asymptotically the Ramanujan approximation is better than 
the Landau one, if $\gamma_S>\frac12$ it is the other way around.
\end{Cor}
\noindent
For more details on the ``Landau versus Ramanujan'' theme, see Moree \cite{LvR}.

A large class of multiplicative sets is provided by the family of sets
\begin{equation}
    \label{setdefinition}
S_{f,q} =\{n:~q\nmid f(n)\},
\end{equation}
where $q$ is any prime and $f$ any integer-valued multiplicative function.
If $S$ is a set that in some sense is close to the set of all natural numbers, then $\gamma_S$ will be close to 
$\gamma$. This will be  
for example the case if $q$ in \eqref{setdefinition} is a large prime.
The following example demonstrates this.
\begin{Ex}
Put $S_{\varphi,q}=\{n:~q\nmid \varphi(n)\}$, with $q\ge 3$ a prime.  
Ford et al.\,\cite{FLM} established that asymptotically
$\gamma_{S_{\varphi,q}}=
\gamma+O_{\varepsilon}(q^{\varepsilon-1})$ 
and related $\gamma_{S_{\varphi,q}}$ to the Euler constants of the cyclotomic fields $\mathbb Q(\zeta_q)$.  
\end{Ex}
\begin{Ex}
The generalized divisor 
sum is defined by $\sigma_k(n)=\sum_{d\mid n}d^k$, where $k$ is any integer. 
Ciolan et al.\,\cite{CLM} studied the doubly infinite family 
$S_{\sigma_k,q}=\{n:~q\nmid \sigma_k(n)\}$, with $k$ an arbitrary integer and $q$ 
an arbitrary prime. They gave an explicit formula for $\gamma_{S_{\sigma_k,q}}$ 
and determined the first and second order behavior of $S_{\sigma_k,q}(x)$, 
refining work of Rankin \cite{Rankin61}, who had deterimined the asymptotic behavior. 
The motivation for studying the divisor sums is that 
there are many interesting congruences relating 
them to Fourier coefficients, the most famous being
$\tau(n)\equiv \sigma_{11}(n)\pmod*{691}$, where $\tau$ denotes Ramanujan's tau-function.
\end{Ex}
The sets $S$ in the examples above, and indeed in many considered in the literature, are abelian.
 \begin{Def}
A multiplicative set $S$ is called \textbf{abelian} if the prime numbers in it, with
finitely many exceptions, are in a finite union of arithmetic progressions.
\end{Def}
By the Chinese Remainder Theorem we can find an integer $d$ such that the primes are, 
with finitely many exceptions, those in a number of primitive residue classes modulo $d$. 
Corollary \ref{Corollary-abelian} then shows 
that the computation of $\gamma_S$ is reduced to computation of
$\gamma(d,a)$ for appropriate $a$.

We note that $S_{f,q}$ is abelian if $f(p)$ is polynomial in $p$, e.g., $\varphi(p)=p-1$ and $\sigma_k(p)=1+p^k$.

\section{Properties of \texorpdfstring{$\gamma(d,a)$}{gda} and some special cases}
\label{sec:properties-specialcases}
\subsection{Properties of \texorpdfstring{$\gamma(d,a)$}{gda}}
\label{sec:gammaproperties}
On noting that 

$$
\prod_{\substack{a=1\\ (d,a)=1}}^{d-1}\zeta_{\,d,a}(s)=\zeta(s)\prod_{p\mid d}(1-p^{-s}),
$$
by taking the logarithmic derivative of both sides and letting $s\to 1^+$, we obtain
\begin{equation}
\label{gammasum}    
\sum_{\substack{a=1\\ (d,a)=1}}^{d-1}\gamma(d,a)=\gamma + \sum_{p\mid d}\frac{\log p}{p-1}.
\end{equation}
This identity also follows directly from the limit result \eqref{nicelimit}. 
The sum on the right hand side is considered in extenso in Hashimoto et al.\,\cite{HKT}.

We assume that $d>1$ is \emph{odd}. It is easy
to relate $\gamma(d,a)$ with an Euler constant for modulus $2d$.
It is a consequence 
of \eqref{nicelimit} that
\begin{equation}
\label{twiced-relation}    
\gamma(d,2)=\gamma(2d,2+d)-\log 2.
\end{equation}
If $a\ne 2$, then
\begin{equation}
\label{twocases}    
\gamma(d,a)=
\begin{cases}
\gamma(2d,a)& \text{~if~}2\nmid a;\\
\gamma(2d,a+d)& \text{~if~}2\mid a.
\end{cases}
\end{equation}
In this way, for example, one gets $\gamma(10,1)=\gamma(5,1)$,
$\gamma(10,3)=\gamma(5,3)$,  $\gamma(10,7)=\gamma(5,2)+\log 2$
and  $\gamma(10,9)=\gamma(5,4)$.

For $d=1$ identity \eqref{twocases} does
not hold; then
we have $\gamma(2,1) =\gamma(1,1) + \log 2=\gamma+\log 2$. 
These identities together with \eqref{twocases} show that in considerations 
of $\gamma(d,a)$, without loss of generalization we may assume that $d\not\equiv 2\pmod*{4}$.

\subsection{The case \texorpdfstring{$a=1$}{a=1}.}
For $d\ge 3$ we obtain 
\begin{align}
\gamma(d,1)
\label{gamma-a1}
= 
\gamma_1(d,1)
+
\sum^{d-1}_{\substack{b=2 \\(d,b)=1}}
\sum_{p\equiv b \pmod* d}  \frac{\log p}{p^{\nu_d(b)}-1}
=\gamma_1(d,1)
-
\sum^{d-1}_{\substack{b=2 \\(d,b)=1}}
\frac{\zeta'_{\,d,b}}{\zeta_{\,d,b}}(\nu_d(b)),
\end{align}
The first equality follows from Theorem \ref{mainabelian} and \eqref{compacttwo} on
noting that $\nu_d(1)=1$, $e_1=2$ and $\nu_d(b)=e_b$ for every $b\ge 2$ 
satisfying $(d,b)=1$.
The second equality immediately follows by 
using the definition of $\zeta_{d,a}(s)$ to get
\[
\log\zeta_{\,d,a}(s)=-\sum_{p\equiv a\pmod*{d}} \log(1-p^{-s}),
\]
and then by differentiating both sides.

\subsection{The case \texorpdfstring{$d=3$}{d3}.}
\label{sec:d3}
By \eqref{gamma-a1}
we immediately have
\begin{align*}
\gamma(3,1) & 
= \gamma_1(3,1)
+ \!\!\!\sum_{p \equiv 2 \pmod*{3}} \frac{\log p}{p^2-1}
=
 \gamma -  \frac{\log 3}{2} -3 \log\Gamma\bigl(\frac13\bigr) + 2\log(2\pi) 
 + \!\!\!\sum_{p \equiv 2 \pmod*{3}} \frac{\log p}{p^2-1},
 \end{align*}
 where $\Gamma$ denotes Euler's gamma function, $\gamma_1(3,1)$ is defined in \eqref{gamma1-def},
 $\chi_{-3}$ denotes the quadratic character modulo $3$ and we used
 \[
\frac{L'}{L}(1,\chi_{-3})
=
\gamma-6\log\Gamma\bigl(\frac13\bigr) -\frac{3}{2} \log 3 + 4\log(2\pi),
\]
which can be obtained using classical formulae, see, e.g., (8) and 
(10)-(12) in \cite{Languasco2021a}, or \cite{LanguascoR2021}.
Moreover, recalling \eqref{gammasum}, we also obtain 
 \begin{align*}
\gamma(3,2) 
&=  
\gamma_1(3,2)
- \sum_{p \equiv 2 \pmod*{3}} \frac{\log p}{p^2-1}
 =
 \log 3 + 3 \log\Gamma\bigl(\frac13\bigr) 
 - 2\log(2\pi)- \sum_{p \equiv 2 \pmod*{3}} \frac{\log p}{p^2-1},
\end{align*}
where $\gamma_1(3,2)$ is defined in \eqref{gamma1-def}.
The prime sum  in the previous formulae for $\gamma(3,1), \gamma(3,2)$
can be computed with a large precision by using 
\[
\sum_{p\equiv 2\pmod*{3}}\frac{\log p}{p^{2}-1}
=
\frac{1}{2} 
\sum_{j=1}^{J}\bigg(
\frac{L'}{L}(2^j,\chi_{-3})
-
\frac{\zeta'}{\zeta}(2^j)
-\frac{\log 3}{3^{2^j}-1}\bigg)
+
\sum_{p\equiv 2 \pmod*3} \frac{\log p}{p^{2^{J+1}}-1}
\]
which follows from recursively applying $J$ times the result in \cite[eq.~(64)]{CLM}.
Letting $\Delta\in \mathbb N$, it is not hard to prove that the 
last sum in the previous formula is $<10^{-\Delta}$ for 
\[J > \bigl(\Delta\log 10 + \log(\log 4 + 4 \log 3)\bigr)/\log 4.\]
Moreover, 
%
using again \cite[eq.~(66)]{CLM}, we can also obtain
\begin{equation*}
\frac{L'}{L}(2^j,\chi_{-3})
=
-\log 3 + \frac{\zeta^\prime (2^j,1/3) - \zeta^\prime (2^j,2/3) }
{ \zeta(2^j,1/3)-\zeta(2^j,2/3)},
\end{equation*}
where $\zeta(s,x)$ denotes the Hurwitz zeta function 
and $\zeta^\prime(s,x)=\frac{\partial}{\partial s}\zeta(s,x)$.
These formulae allowed us to compute both $\gamma(3,1)$ and $\gamma(3,2)$ with  
$130\, 000$ decimals\footnote{The computation required about 16 days 
using 240 GB of RAM on a 4 x Eight-Core Intel(R) Xeon(R) CPU E5-4640 @ 2.40GHz
machine of the cluster of the Dipartimento di Matematica of the University of Padova.}.
Moreover, see \urltwo{https://oeis.org/A073005}{OEIS-A073005},
\[
\Gamma\bigl(\frac{1}{3}\bigr) 
=
\frac{2^{7/9} \pi^{2/3}}{3^{1/12} \bigl(M\bigr(2, \sqrt{2+\sqrt{3}}\bigr)\bigr)^{1/3}},
\]
where $M(a,b)$ (with $a,b>0$) denotes the \emph{arithmetic-geometric mean} (AGM) of $a$ 
and $b$.
Using the AGM algorithm
one can obtain $\Gamma(\frac13)$ with a billion decimals
in less than a hour of computation time.
On the \urltwo{http://www.numberworld.org/digits/}{numberworld\_website} it is reported 
that $2\cdot 10^{11}$ decimals of $\Gamma(\frac13)$ are currently known.

\subsection{The case \texorpdfstring{$d=4$}{d4}}
\label{sec:d4} This case is closely related to the 
historically important example where $S$ is 
taken to be the set of sums of two squares, cf.\,Berndt and Moree \cite{BeMo}.
Since $S$ is generated by the prime $2$, the primes $1\pmod*{4}$ and the squares of the primes $3\pmod*{4}$, 
for $\Re(s) > 1$ we obtain
\begin{equation}
\label{basicgenerating}
L_S(s)= (1-2^{-s})^{-1} \zeta_{\,4,1}(s) \zeta_{\,4,3}(2s). 
\end{equation}
Denoting by $\chi_{-4}$ the quadratic character modulo $4$ and recalling
that
$$L(s,\chi_{-4})=\prod\limits_{p\equiv 1 \pmod*{4}}(1-p^{-s})^{-1}
\prod\limits_{p\equiv 3 \pmod*{4}}(1+p^{-s})^{-1},$$
one verifies that
\begin{equation*}
L_S(s)^2=\zeta(s)L(s,\chi_{-4})(1-2^{-s})^{-1}
\zeta_{\,4,3}(2s),
\end{equation*}
leading to
\begin{equation}
\label{theanswer}    
2\gamma_S=\gamma+\frac{L'}{L}(1,\chi_{-4})-\log 2-\sum_{p\equiv 3\pmod*{4}}\frac{2\log p}{p^{2}-1}.
\end{equation}
By the prime number theorem for arithmetic 
progressions \eqref{primecondition} is satisfied with
 $\delta=\frac12$ and hence 
 $c_1(S)=(1-\gamma_S)/2$ by Theorem \ref{thm:multiplicativeset},
 which
on combination with \eqref{theanswer} gives a
formula for $c_1(S)$ due to Shanks \cite{Shanks}. 
He showed that $c_1(S)\approx 0.581948659$ and hence
$\gamma_S\approx -0.163897318$.
The constant $c_1(S)$ is now called \emph{Shanks' constant}
(see also \urltwo{https://oeis.org/A227158}{OEIS-A227158}).
Currently it can be computed with much higher precision: the
record is \numSconstantdecimals decimal digits \cite{Sconstant} (due to the first author).

Several mathematicians in the 19th century independently discovered the identity
\begin{equation*}
\frac{L'}{L}(1,\chi_{-4})=\gamma-\log 2-2\log G,
\end{equation*}
where $$G=\frac{1}{ M(1,\sqrt{2})}=\frac{2}{\pi}\int_0^1 \frac{dx}{\sqrt{1-x^4}}=0.8346268416740731862814297\dotsc$$
The constant $G$ is now named 
\emph{Gauss's constant} and is
also related to various other constants
\cite{Gaussconstant}.
We infer that 
\eqref{theanswer} can be rewritten as
\begin{equation}
\label{theanswer2}    
\gamma_S=\gamma-\log G-\log 2-\sum_{p\equiv 3\pmod*{4}}\frac{\log p}{p^{2}-1}.
\end{equation}

As shown in Ciolan et al.\,\cite[Section~10.2]{CLM}
one has
\[
\sum_{p\equiv 3\pmod*{4}}\frac{\log p}{p^{2}-1}
=
\frac{1}{2} 
\sum_{j=1}^{J}\bigg(
\frac{L'}{L}(2^j,\chi_{-4})
-
\frac{\zeta'}{\zeta}(2^j)
-\frac{\log 2}{2^{2^j}-1}\bigg)
+
\sum_{p\equiv 3 \pmod*4} \frac{\log p}{p^{2^{J+1}}-1},
\]
which allows one to compute the left hand term with high accuracy 
by choosing $J$ sufficiently large and using that
\begin{equation}
\label{quadlogder-at2^k}
\frac{L'}{L}(2^j,\chi_{-4})
=
-2\log 2 + \frac{\zeta^\prime (2^j,1/4) - \zeta^\prime (2^j,3/4)}
{\zeta(2^j,1/4)-\zeta(2^j,3/4)}
\end{equation}
which follows from \cite[eq.~(66)]{CLM}
(which is valid for every integer $q\ge 3$).

In fact the right hand side of \eqref{quadlogder-at2^k}
equals the logarithmic derivative at $2^j$ of the \emph{Dirichlet $\beta$-function}
$\beta(s) = 4^{-s}(\zeta(s,\frac14)-\zeta(s,\frac34))$, $\Re(s) > 1$.

Exploiting the computations described in Languasco \cite[Section~10]{CLM}, see also \cite{Sconstant}, 
we were able compute $130\, 000$ decimals for $G$.
But using directly the AGM algorithm one can obtain $G$ with $10^6$ of decimals 
in less than a second of computation time and with $10^9$ of decimals in about
$40$ minutes of computation time\footnote{Results obtained using Pari/GP v.2.15.5
on an Intel i7-13700KF machine, with 64GB of RAM.}. Moreover, recalling 
that $G=\Gamma(\frac14)^2(2\pi)^{-3/2}$, and exploiting that $31\,415\,926\,535\,897$ decimals of $\pi$ 
and $362\,560\,990\,822$ of $\Gamma(\frac14)$ 
are known (see the \urltwo{http://www.numberworld.org/digits/}{numberworld\_website}),
we can obtain the same number of decimals for $G$ too.
For more information on $G$, see also \urltwo{https://oeis.org/A014549}{OEIS-A014549}.

{}From \eqref{basicgenerating} we infer that 
$$\gamma_S=-\log 2 + \gamma(4,1)-\sum_{p\equiv 3\pmod*{4}}\frac{2\log p}{p^{2}-1}.$$
This in combination with \eqref{theanswer2} 
and \eqref{gammasum} yields
$$\gamma(4,1)=\gamma-\log G +\sum_{p\equiv 3\pmod*{4}}\frac{\log p}{p^{2}-1},
\quad 
\gamma(4,3)=\log G + \log 2 -\sum_{p\equiv 3\pmod*{4}}\frac{\log p}{p^{2}-1}.
$$
{}From  our remarks above it is clear that both
can be obtained with $130\,000$ decimal precision.

\section{How to compute \texorpdfstring{$\gamma(d,a)$}{gda} in the general case}
\label{sec:algorithm}

The following proposition expresses $\gamma(d,a)-\gamma_1(d,a)$ in terms
of the logarithmic derivatives of Dirichlet $L$-function at integral
points $m$, $m\ge 2$.
By $\nu(\chi)$ we will denote the order of the
Dirichlet character $\chi$ as an element of the multiplicative
group of Dirichlet characters having modulus $d$. It is easy to see
that $\nu(\chi^j)= \nu(\chi)/ (j, \nu(\chi))$ for every $j\ge 1$.
\begin{Prop}
\label{Prop-Sk-moebius}
Let $a$ and $d\ge 2$ be coprime integers.
Then
\begin{equation}
\label{Sk-moebius}
\gamma(d,a) 
=
\gamma_1(d,a)
- \frac{1}{\varphi(d)}
\sum_{\chi\ne \chi_0}{\overline{\chi}}(a) 
\sum_{j \ge 1}
\mu(j)\!\!\!
\sum_{\substack{k \ge 2 \\ k \not\equiv 1 \pmod*{\nu(\chi^j)}}}
\!\!\!
\Bigl(
\frac{L'}{L}(kj,\chi^{kj})
-
\frac{L'}{L}(kj,\chi^{j})
\Bigr).
\end{equation}
\end{Prop}
\begin{proof}
Let $s$ be a complex number with $\Re(s)>1$.
Using the Euler product for $L(s,\chi)$ and the Taylor formula we have
$$
\frac{1}{\varphi(d)}\sum_{\chi}{\overline{\chi}}(a)\log L(s,\chi)
=
\frac{1}{\varphi(d)}\sum_{\chi}{\overline{\chi}}(a)\sum_p\sum_{k \ge 1}\frac{\chi(p)^k}{kp^{ks}}.
$$
Moreover, by the Taylor formula and the orthogonality of Dirichlet characters, we also have
\begin{equation}
\label{logzeta_da_taylor}
\log \zeta_{\,d,a}(s) =
\frac{1}{\varphi(d)}\sum_{\chi}{\overline{\chi}}(a)\sum_p\sum_{k \ge 1}\frac{\chi(p)}{kp^{ks}}.
\end{equation}
Substraction gives
\[
\log \zeta_{\,d,a}(s)
-
\frac{1}{\varphi(d)}\sum_{\chi}{\overline{\chi}}(a)\log L(s,\chi)
=
-
\frac{1}{\varphi(d)}\sum_{\chi \ne \chi_0}{\overline{\chi}}(a)
\sum_p\sum_{k \ge 2}
\frac{\chi(p)^k-\chi(p)}{kp^{ks}},
\]
since the contribution  of the cases
$\chi=\chi_0$ and $(\chi\ne \chi_0$ and $k=1)$ are clearly both zero.\footnote{A similar expression was  obtained in \cite[Section~2]{lanzac10} 
in the context of  Mertens'
constants for arithmetic progressions.}
Differentiating and letting $s\to 1^+$,
we have
\begin{equation}
\label{Sk-explicit}
\gamma(d,a) 
-
\gamma_1(d,a)
=
\frac{1}{\varphi(d)}  
\sum_{\chi\ne \chi_0}{\overline{\chi}}(a) 
\sum_p\sum_{k \ge 2} \frac{\log p}{p^k} \bigl(\chi(p)^k-\chi(p) \bigr). 
\end{equation}
The inner double series in \eqref{Sk-explicit} is absolutely convergent 
and hence we can swap the order of summation.

We now recall that, by the M\"obius inversion formula, see, e.g., 
Belabas-Cohen \cite[Sect.\,5.5]{BelabasC2021},
for $\Re(s)>1$ we obtain 
\begin{equation}
\label{inversion}
  \sum_p \frac{\psi(p)}{p^s}
  =
  \sum_{j \ge 1}
    \frac{\mu(j)}{j} \log(L (sj, \psi^j)),
\end{equation}
where $\psi$ is a Dirichlet character modulo $d$.
Differentiating both sides of \eqref{inversion} with respect to $s$ yields
\begin{equation}
\label{Moebius-inversion-general}
\sum_{p} \frac{\psi(p) \log p}{p^s}
=
- \sum_{j\ge 1} \mu(j) \frac{L'}{L}(sj,\psi^{j}).
\end{equation} 
Inserting \eqref{Moebius-inversion-general} into \eqref{Sk-explicit}
twice (the first time using $s=k$ and $\psi = \chi^k$ and the second time using
$s=k$ and $\psi = \chi$), formula \eqref{Sk-moebius} follows immediately
by swapping again the order of summation and remarking that if
$k \equiv 1 \pmod*{\nu(\chi^j)}$, then $\chi^{kj} = \chi^j$.
\end{proof}

\subsection{An algorithm for computing \texorpdfstring{$\gamma(d,a)$}{gda}}
\label{sec:algorithm-bis}
In the general case, unfortunately, Shanks' acceleration procedure mentioned
in Sections \ref{sec:d3}-\ref{sec:d4} cannot be used. 
However, we can exploit the ideas from the proof
of Proposition \ref{Prop-Sk-moebius}, but now we also need to truncate the involved sums.
In fact, this procedure is the analogue of the one used in \cite{lanzac09,lanzac10b}
to evaluate Mertens' constants in arithmetic progressions.

The starting point is \eqref{Sk-explicit}, which for readability we repeat here:
\begin{align}
\notag
\gamma(d,a)
= \gamma_1(d,a)
+
\frac{1}{\varphi(d)}
\sum_p\sum_{k \ge 2}
\frac{\log p}{p^k}\sum_{\chi\ne \chi_0}{\overline{\chi}}(a) 
\bigl(\chi(p)^k-\chi(p) \bigr)
= \gamma_1(d,a) + S_1(d,a),
\end{align}
say.
The computation of $\gamma_1(d,a)$ is trivial, since it is  
a finite sum and hence it is enough to compute 
$L'/L(1,\chi)$ for $\chi \ne \chi_0 \pmod*{d}$,
in order to obtain its value\footnote{The Pari/GP implementation of Dirichlet $L$-functions works only
with primitive characters. In case $\chi \pmod*{d}$ is imprimitive,  
one needs to first find the associated primitive character $\chi^* \pmod*{f}$ (where $f\mid d$), and then 
insert a suitable correction term that involves
the values of $\chi^*$. For the logarithmic derivative $L'/L(s,\chi)$, $\chi\ne \chi_0$, 
$\Re(s) \ge 1$, such a term is
$\sum_{p \mid d} \frac{\chi^*(p)\, \log p}{p^s - \chi^*(p)}$.}, 
which we do exploiting the $L$-functions implementation of Pari/GP.

We explain now how to compute $S_1(d,a)$.
Let $P$ be a large integer to be specified later. 
Splitting the sum over $p$ in $S_1(d,a)$
we obtain
\begin{align}
\notag
S_1(d,a) &= 
\frac{1}{\varphi(d)}\
\sum_{p\le P}\sum_{k \ge 2}  
\frac{\log p}{p^k}\sum_{\chi\ne \chi_0}{\overline{\chi}}(a) 
\bigl(\chi(p)^k-\chi(p) \bigr)
\\
\notag
&\hskip2cm
+
\frac{1}{\varphi(d)}\
\sum_{p > P}\sum_{k \ge 2}
\frac{\log p}{p^k}\sum_{\chi\ne \chi_0}{\overline{\chi}}(a) 
\bigl(\chi(p)^k-\chi(p) \bigr)
\\
\notag
&=
S_2 (d, a, P) + S_3(d, a, P) ,
\end{align}
say. 
Using the geometric series of the ratios $\chi(p)/p$ and $1/p$,
it is not hard to see that 
\begin{equation}
\label{S3-comput}
S_2 (d, a, P) = 
\frac{1}{\varphi(d)}  
\sum_{\chi\ne \chi_0}{\overline{\chi}}(a) 
\sum_{p\le P} \frac{ \chi(p)\, \log p}{p}
\Bigl( 
\frac{\chi(p)}{p-\chi(p)}
-
\frac{1}{p-1}
\Bigr)
\end{equation}
can be easily computed.

In order to evaluate $S_3(d, a, P)$, we need to split the sum over $k$. 
Letting $K$ be a sufficiently large integer, we can write
\begin{align}
\notag
S_3(d, a, P)
&= 
\frac{1}{\varphi(d)}\
\sum_{p > P}\sum_{2 \le k \le K}
\frac{\log p}{p^k}  
\sum_{\chi\ne \chi_0}{\overline{\chi}}(a) 
\bigl(\chi(p)^k-\chi(p) \bigr)
\\
\notag
&\hskip2cm
+
\frac{1}{\varphi(d)}\
\sum_{p > P}\sum_{k > K}
\frac{\log p}{p^k} 
\sum_{\chi\ne \chi_0}{\overline{\chi}}(a) 
\bigl(\chi(p)^k-\chi(p) \bigr)
\\
\label{truncate-K}
&=
S_4(d, a, P) + E_1(d, a, P, K) ,
\end{align}
say. The error term $E_1(d, a, P, K)$ can be estimated as
\begin{align}
\notag
\vert E_1(d, a, P, K) \vert 
&\le
2\frac{(\varphi(d)-1)}{\varphi(d)}\sum_{p > P} \log p \sum_{k > K}\frac{1}{p^k} 
<
2\sum_{p > P} \frac{\log p}{p^K(p-1)}
\\
\label{E1-estim}&
<
2\frac{\log P}{P-1}\int_{P}^{+\infty} \frac{dt}{t^K}
<\frac{2}{(K-1)} \frac{\log P}{P^{K-1}(P-1)}.
\end{align}

We are left with estimating $S_4(d, a, P)$. 
We define
\begin{equation}
\label{truncate-P}
L_P(s,\chi) = L(s,\chi) \prod_{p\le P} \Big(1-\frac{\chi(p)}{p^s}\Big),\quad \Re(s)>1,
\end{equation}
and note that
%
\begin{equation}
\label{derlog-explicit}
\frac{L'_P}{L_P}(s,\chi)
=
\frac{L'}{L}(s,\chi)
+
\sum_{p\le P} \frac{\chi(p) \, \log p}{p^s - \chi(p)}.
\end{equation}
Arguing as in the proof of Proposition \ref{Prop-Sk-moebius},
we obtain 
\begin{equation}
\label{Moebius-inversion}
\sum_{p > P} \frac{\log p}{p^s}\bigl(\chi(p)^k-\chi(p)\bigr)
=
-\sum_{j\ge 1} \mu(j)
\Bigl(
\frac{L'_P}{L_P}(sj,\chi^{kj})
-
\frac{L'_P}{L_P}(sj,\chi^{j})
\Bigr).
\end{equation} 
Observe that if $k \equiv 1 \pmod*{\nu(\chi^j)}$ or $j$ is not a square-free number, then the 
corresponding summand 
at the right hand side of \eqref{Moebius-inversion} is zero.

Let  $J$ be a sufficiently large integer. Inserting  \eqref{Moebius-inversion} 
into $S_4(d, a, P)$ and splitting the sum over $j$, we have
\begin{align}
\notag
S_4(d, a, P)
&=
-\frac{1}{\varphi(d)}\ 
\!\!\!
\sum_{\chi\ne \chi_0}{\overline{\chi}}(a)  
\sum_{1 \le j \le J}
\mu(j)\!\!\! 
\sum_{\substack{2\le k \le K \\ k \not\equiv 1 \pmod*{\nu(\chi^j)}}}
\Bigl(
\frac{L'_P}{L_P}(kj,\chi^{kj})
-
\frac{L'_P}{L_P}(kj,\chi^{j})
\Bigr)
\\
\notag
&\hskip1cm
-\frac{1}{\varphi(d)}\ 
\sum_{\chi\ne \chi_0}{\overline{\chi}}(a) 
\sum_{2 \le k \le K}
\sum_{j > J}
\mu(j)
\Bigl(
\frac{L'_P}{L_P}(kj,\chi^{kj})
-
\frac{L'_P}{L_P}(kj,\chi^{j})
\Bigr)
\\
\label{truncate-J}
&
=S_5(d, a, P, K, J) + E_2 (d, a, P, K, J),
\end{align}
say. 
Remark that $S_5(d, a, P, K, J)$ contains only finite sums;
the values of $L'_P/L_P(s,\chi)$ occurring there
can be evaluated  exploiting the $L$-functions implementation of Pari/GP
and \eqref{derlog-explicit}. 

We now estimate $E_2 (d, a, P, K, J)$.
Using, for $\sigma=\Re(s) \ge 3$, that

\[
\Bigl\vert \frac{L'_P}{L_P}(s,\chi) \Bigr\vert
=
\Bigl\vert 
\sum_{p > P} \frac{\chi(p) \, \log p}{p^s - \chi(p)}
\Bigr\vert 
\le
\sum_{p > P} \frac{\log p}{p^\sigma - 1}
<
\sum_{p > P} \frac{1}{p^{\sigma - 1}}
<
\int_{P}^{+\infty} \frac{dt}{t^{\sigma - 1}}
= 
\frac{P^{2-\sigma}}{\sigma -2},
\]
we obtain
\begin{align}
\notag
\vert E_2 (d, a, P, K, J)\vert 
&\le
2\frac{(\varphi(d)-1)}{\varphi(d)}\sum_{2 \le k \le K}\sum_{j>J} \frac{P^{2- kj}}{kj -2}
<
2 P^2\sum_{2 \le k \le K}\frac{1}{k}\sum_{j>J} \frac{1}{P^{kj}}
\\
&
= 2 P^2\sum_{2 \le k \le K}\frac{1}{k(P^k-1)} \frac{1}{P^{kJ}}
<
2 P^2\sum_{k \ge 2} \frac{1}{P^{k(J+1)}}
\label{E2-estim}
= \frac{2}{P^{J-1}(P^{J+1}-1)}.
\end{align}

Summarising, we have proved that 
\[
\gamma(d,a) 
=
\gamma_1(d,a) +S_2 (d, a, P) + S_5(d, a, P, K, J) + E_1(d, a, P, K) + E_2 (d, a, P, K, J),
\]
where these quantities are defined and estimated, respectively, in \eqref{gamma1-def},
\eqref{S3-comput}, \eqref{truncate-J}, \eqref{truncate-K} and \eqref{truncate-J} again.

Let $\Delta \ge 2$ be an integer.
The estimates  \eqref{E1-estim} and \eqref{E2-estim}
allow us to ensure that 
$\vert E_1(d, a, P, K) \vert$ $+ \vert E_2 (d, a, P, K, J)\vert < 10^{-\Delta}$
for every $d\ge 3$, $a=1,\dotsc,d-1$ coprime to $d$ by choosing $P,K,J$ sufficiently large.
In order to obtain a good approximate value for $\gamma(d,a)$, it
is thus enough to compute the finite sums $\gamma_1(d,a)$, $S_2 (d, a, P)$ and $S_5(d, a, P, K, J)$
with a sufficiently large accuracy.

\subsection{The numerical results}
Using Pari/GP v.2.15.5
on an Intel i7-13700KF machine, with 64GB of RAM, 
we were able to compute the values of $\gamma(d,a)$,
$3\le d\le 100$,  $a=1,\dotsc,d-1$ coprime to $d$, with an accuracy of at least $100$ decimals
(about four hours and eight minutes
of computation time were required);
and $3\le d\le 500$,  $a=1,\dotsc,d-1$ coprime to $d$, with an accuracy of at least $20$ decimals
(four hours).
%
The values for $3\le d\le 1000$,  $a=1,\dotsc,d-1$ coprime to $d$,
computed with an accuracy of at least $10$ decimals required 
fifteen and a half hours.
Obtaining the complete set of results for larger values of $d$ would require a larger amount of memory.
Some of these cases are reported in Table \ref{tab:24}.
The cases having a large accuracy are the ones described in Sections
\ref{sec:d3}-\ref{sec:d4}.

Our computer program and the results
obtained are available on the webpage \url{http://www.math.unipd.it/~languasc/Euler_contants_AP.html}.
The relations \eqref{twiced-relation}-\eqref{twocases} 
were first used only to check
the correctness of the results, but
later on we decided to insert them
into the final computations, since their use leads to a
consistent reduction of the total computational time
(at the cost of requiring a larger amount of memory to store the $\gamma(d,a)$-values
already computed).

\begin{Rem}
\label{Remark-sums}
The ideas in Proposition \ref{Prop-Sk-moebius}
and in Section \ref{sec:algorithm-bis} can be also used, with some minor adaptation, 
to compute the four sums:
\begin{inparaenum}[(a)]
\item 
\label{firstsum}
$\sum_p \frac{\log p}{p^s(p-1)},$ 
\custspace
\item  
\label{secondsum}
$\sum\limits_{p\equiv a \pmod*{d}} \frac{\log p}{p^s(p-1)} \quad \textrm{for}\ \Re(s)>0,$ 
\custspace
\\
\item 
\label{thirdsum}
$\frac{\zeta'_{\,d,a}}{\zeta_{\,d,a}}(s)  = - \sum\limits_{p\equiv a \pmod*{d}}  \frac{\log p}{p^s-1},
\ \textrm{for}\ \Re(s)>1, $
\custspace
\item 
\label{fourthsum}
$\sum\limits_{p\equiv a \pmod*{d}} \frac{p^u\log p}{p^s-1},
\ \textrm{for}\ \Re(s)>\Re(u) + 1.$
\end{inparaenum}

Writing $s=\sigma+it$, $\sigma>0$, 
and respectively denoting with $\lfloor \sigma \rfloor$ and $\{\sigma\}$
the integral and the fractional part of $\sigma$,
the starting point for the evaluation of the sum \eqref{firstsum}
is the formula

\begin{align*}
\sum_p 
\frac{\log p}{p^s(p-1)}
&= 
\sum_p 
\frac{\log p}{p^{\{\sigma\}+it}}
\sum_{\ell\ge\lfloor \sigma \rfloor+1} \frac{1}{p^\ell}
=
\sum_{\ell\ge\lfloor \sigma \rfloor+1}
\sum_p 
\frac{\log p}{p^{\ell+\{\sigma\}+it}}
\\&=
\sum_{k\ge1}
\sum_p 
\frac{\log p}{p^{k+s}}
=
-\sum_{k \ge 1}
\sum_{j\ge 1}  \mu(j)
\frac{\zeta'}{\zeta} \bigl((k+s)j \bigr)
\quad (\Re(s)> 0),
\end{align*} 
which can be obtained by M\"obius inversion (or by using \eqref{Moebius-inversion-general} with $d=1$).

For the analogous sum  \eqref{secondsum}
a similar formula holds true, but in this case the values of $\frac{L'}{L}\bigl((k+s)j, \chi^j\bigr)$
are involved. 
Note that this second sum is present in  \eqref{compacttwo} with $s=\nu_d(b)$ and $e_b=\nu_d(b)$.

For sum  \eqref{thirdsum} 
the main formula is obtained using
\begin{equation*}
\frac{\zeta'_{\,d,a}}{\zeta_{\,d,a}}(s) =
-
\frac{1}{\varphi(d)}\sum_{\chi}{\overline{\chi}}(a)\sum_p\sum_{k \ge 1}\frac{\chi(p)\log p}{p^{ks}},
\end{equation*}
which follows by differentiating \eqref{logzeta_da_taylor}, and then, for every $k\ge 1$, by inserting the relation
\[
\sum_{p} \frac{\chi(p) \log p}{p^{ks}}
=
- \sum_{j\ge 1} \mu(j) \frac{L'}{L}(ksj,\chi^{j}),
\]
which follows from \eqref{Moebius-inversion-general} by taking $\psi=\chi$ 
and replacing $s$ with $ks$. 
One finally obtains
\[
\frac{\zeta'_{\,d,a}}{\zeta_{\,d,a}}(s) =
\frac{1}{\varphi(d)}\sum_{\chi}{\overline{\chi}}(a) \sum_{k \ge 1} \sum_{j\ge 1} \mu(j) \frac{L'}{L}(ksj,\chi^{j})
\quad
(\Re(s)>0).
\]

The  sum \eqref{fourthsum} is present in  \eqref{compacttwo} with
$u=\nu_d(b)-e_b$ and $s=\nu_d(b)$.
In order to evaluate it, it is enough to 
remark that  $\frac{p^u}{p^s-1}= \sum_{k \ge 1} p^{u-ks}$,
and then proceed as for $\frac{\zeta'_{\,d,a}}{\zeta_{\,d,a}}(s)$,
giving
\begin{align*}
\sum_{p\equiv a \pmod*{d}} \frac{p^u\log p}{p^s-1}
&=
 \sum_{k \ge 1} \sum_{p\equiv a \pmod*{d}}  \frac{\log p}{p^{ks-u}}
 =
 \frac{1}{\varphi(d)}\sum_{\chi}{\overline{\chi}}(a)\sum_p
 \sum_{k\ge 1}\frac{\chi(p)\log p}{p^{ks -u}}
 \\&
 =
 -  \frac{1}{\varphi(d)}\sum_{\chi}{\overline{\chi}}(a)
 \sum_{k\ge 1}
  \sum_{j\ge 1} \mu(j) \frac{L'}{L}\bigl((ks-u) j,\chi^{j}\bigr)
  \quad
  (\Re(s)>\Re(u) + 1).
\end{align*}

All the formulae presented in this remark 
could be used to compute $\gamma(d,a)$
using \eqref{compacttwo}; this is not a good idea, though, since in this
way we would not only need to compute the orders of the non-trivial
elements of 
$(\mathbb Z/d\mathbb Z)^*$, but we would also lose the cancellation of the terms corresponding 
to $k=1$ and 
$k\equiv 1 \pmod*{\nu(\chi^j)}$, which we exploited in the algorithm presented
in Section \ref{sec:algorithm-bis}.
\end{Rem}

\begin{Rem}
We also point out that $\zeta_{\,d,a}(s)$, $\Re(s)>1$, can be computed analogously by
making use of the  M\"obius inversion formula (for a more general setting see \cite{ERS, Accurate}).
This could be useful, for example, for evaluating the product in  \eqref{serrebeta}.

We remind ourselves that (see \eqref{logzeta_da_taylor}) 
\begin{equation}
\label{logzeta_da_taylor-2}
\log \zeta_{\,d,a}(s) =
- \sum_{p\equiv a \pmod*{d}} \log(1-p^{-s})
=
\frac{1}{\varphi(d)}\sum_{\chi}{\overline{\chi}}(a)\sum_p\sum_{k \ge 1}\frac{\chi(p)}{kp^{ks}}.
\end{equation}
On applying formula \eqref{inversion}, we then obtain 
\begin{equation}
\label{logzeta_inversion}
\log \zeta_{\,d,a}(s) =
\frac{1}{\varphi(d)}\sum_{\chi}{\overline{\chi}}(a)\sum_{k \ge 1} \frac{1}{k}
\sum_{j \ge 1} \frac{\mu(j)}{j} \log(L (ksj, \chi^j)).
\end{equation}

To transform formula \eqref{logzeta_inversion} in a procedure
capable of computing $\log \zeta_{\,d,a}(s)$ with a prescribed accuracy, 
one then needs, as in the algorithm explained in Section \ref{sec:algorithm-bis},
to accelerate the convergence rate by introducing
the parameter $P$ as in \eqref{truncate-P}, 
and to truncate the sums over $j,k$ into \eqref{logzeta_inversion}.
This can be done by following the procedure in  \cite[\S3]{lanzac09}; one just
needs to replace Lemma 1 there with the following slightly generalised version.
\begin{Lem}
Let $\chi \pmod*{d}$ be a Dirichlet character and $s\in \mathbb C$ such that $\sigma = \Re s >1$.
If $P \ge 1$ is an integer then
\[
  \bigl| \log \bigl( L_P(s,\chi) \bigr) \bigr|
  \le
  \frac{2 P^{1 - \sigma}}{\sigma - 1}.
\]
\end{Lem}

\begin{proof}
By the triangle inequality
\[
  \bigl| \log \bigl( L_P(s,\chi) \bigr) \bigr|
  =
  \Bigl|
    \sum_{p > P}
      \sum_{k \ge 1} \frac{\chi^k(p)}{k p^{ks}}
  \Bigr|
  \le
  \sum_{p > P}
    \sum_{k \ge 1} \frac1{k p^{k \sigma}}
  \le
  \sum_{m > P} \frac1{m^\sigma-1}
  \le
  2
  \int_P^{+\infty} \frac{d t}{t^\sigma}
  =
  \frac{2P^{1 - \sigma}}{\sigma - 1},
\]
as required.
\end{proof}

Another  example in which this procedure is useful is in dealing with sums of the form
\[
\sum_{p\equiv a \pmod*{d}} \log(1+p^{-s})
=
\frac{1}{\varphi(d)}\sum_{\chi}{\overline{\chi}}(a)\sum_p\sum_{k \ge 1}\frac{(-1)^{k+1}\chi(p)}{kp^{ks}}.
\]
\end{Rem}

\section{An application: Euler constants of abelian number fields}
\label{sec:Dedekind}
Let $K$ be a number field, ${\mathcal O}$ its ring of integers and $s$ a complex variable.
Then for $\Re(s)>1$ the 
\emph{Dedekind zeta function} is defined as

\begin{equation*}
\zeta_K(s)=\sum_{\mathfrak{a}} \frac{1}{N{\mathfrak{a}}^{s}}
=\prod_{\mathfrak{p}}\frac{1}{1-N{\mathfrak{p}}^{-s}},
\end{equation*}
where $\mathfrak{a}$ ranges over 
the non-zero ideals in ${\mathcal O}$, 
$\mathfrak{p}$ ranges over the prime ideals in ${\mathcal O}$, and $N{\mathfrak{a}}$ denotes the 
\emph{absolute norm}
of $\mathfrak{a}$,
that is the index of $\mathfrak{a}$ in $\mathcal O$.
It is known that $\zeta_K(s)$ can be analytically continued to ${\mathbb C} \setminus \{1\}$,
and that it has a simple pole at $s=1$. 
The Euler constant  $\gamma_K$ of the number 
field $K$ (sometimes called Euler-Kronecker constant in the literature), is 
\begin{equation*}
\gamma_K=\lim_{s \to 1^+}\,\Bigl(\frac{\zeta'_K(s)}{\zeta_K(s)}+\frac{1}{s-1}\Bigr).
\end{equation*}
It can be written as a simple explicit expression involving a sum over the reciprocals of the non-trivial zeros of $\zeta_K(s)$, cf.\,Ihara \cite[pp.\,416-421]{Ihara}.
The following expression, cf.\,\eqref{nicelimit}, can be used to obtain a quick rough approximation
of $\gamma_K$ (for a proof see, e.g., Hashimoto et al.\,\cite{HIKW}):
\begin{equation*}
\gamma_K=\lim_{x\rightarrow \infty}\Bigl(\log x - \sum_{N\mathfrak{p}\le x}\frac{\log N\mathfrak{p}}{N\mathfrak{p}-1}\Bigr).
\end{equation*}
The terms in the sum involving the
norms of non-completely split primes converge and so the crucial contributions to the limit
comes from the completely splitting primes. In case $K$ is abelian, then by the 
Kronecker-Weber theorem $K$
is contained in some cyclotomic number field $\mathbb Q(\zeta_d)$, where we take $d$ to be
minimal. With finitely many exceptions, the primes that split completely are
precisely those in a union of arithmetic progressions modulo $d$. 
The upshot is that $\zeta_K(s)$ is of the form \eqref{abelianformat}   
and $\gamma_K$ satisfies an identity of 
the type \eqref{generalident}. In the particular case that $K=\mathbb Q(\zeta_d)$ we obtain 
$\gamma_K=\varphi(d)\,\gamma(d,1)+\frac{H'}{H}(1)$. Recall that
$\gamma(d,1)$ can be evaluated using \eqref{gamma-a1}. 
In case $d=q$ is a prime we obtain
\begin{equation*}
\gamma(q,1)=
\frac{1}{q-1}\Bigl(\gamma_q+\frac{\log q}{q-1} \Bigr)
+ 
\sum_{b = 2}^{q-1}
\sum_{p\equiv b \pmod* q}  \frac{\log p}{p^{\nu_d(b)}-1}.
\end{equation*}

There is an extensive literature on the average behavior of $\gamma_K$ in certain families, e.g.: \cite{FLM, Fouvry2013} (cyclotomic), \cite{MourtadaM2013} (quadratic).

For number fields there is also an analogue of Mertens' theorem, see, e.g., Rosen \cite{Rosen}, namely we have
$$\prod_{N\mathfrak{p}\le x}\Big(1-\frac{1}{N\mathfrak{p}}\Big)=\frac{e^{-\gamma}}{\alpha_K \log x}\Big(1+O_K\Big(\frac{1}{\log x}\Big)\Big),$$
where $\alpha_K$ is the residue of $\zeta_K(s)$
at $s=1$. This can be used for example to determine $C(d,1)$, see Ford et al.\,\cite[\S 2.3]{FLM}.

\begin{Ex}
We consider the number field $K=\mathbb Q(\zeta_7+\zeta_7^{-1})$ of degree 3.
By abuse of notation we will write $\zeta_{\,d,\pm a}(s)$ for
$\zeta_{\,d,a}(s)\,\zeta_{\,d,-a}(s)$. 
Elementary algebraic number theory shows that
$$\zeta_K(s)=(1-7^{-s})^{-1}\zeta_{\,7,\pm 1}(s)^3\zeta_{\,7,\pm 2}(3s)\zeta_{\,7,\pm 3}(3s).$$
Logarithmic derivation yields
\begin{equation}
\label{gammaK7}
\gamma_K=-\frac16\log 7+3\,\gamma(7,1)+3\,\gamma(7,6)-\sum_{p\equiv \pm 2,\pm 3 \pmod* 7}\frac{3\log p}{p^3-1}.
\end{equation}
The values of $\gamma(7,1)$ and $\gamma(7,6)$ 
can be computed with $100$ decimals, see Section \ref{sec:algorithm} and Table \ref{tab:24}. 
The prime sum is equal to
\[
- 3 
\Bigl(\frac{\zeta'_{\,7, 2}}{\zeta_{\,7, 2}}(3)
+
\frac{\zeta'_{\,7, 3}}{\zeta_{\,7, 3}}(3)
+
\frac{\zeta'_{\,7, 4}}{\zeta_{\,7, 4}}(3)
+
\frac{\zeta'_{\,7, 5}}{\zeta_{\,7, 5}}(3)
\Bigr)
\]
and hence it can be computed as described in Remark \ref{Remark-sums}
with the very same accuracy of 100 decimals.
The first forty digits are 
$\gamma_K=1.957156454449714752713821861425456626477\dotsc$

A way to verify the correctness of \eqref{gammaK7} is to insert in it the formulae
for $\gamma(7,1)$ and $\gamma(7,6)$ (stated in Section \ref{sec:d7}), 
leading to
$$\gamma_K = 
\gamma + \frac12 
\sum_{\chi\ne \chi_0}
(1+ \chi(-1)) \frac{L'}{L}(1,\chi) 
=
\gamma + \sum_{\substack{\chi(-1) = 1 \\ \chi\ne \chi_0}} \frac{L'}{L}(1,\chi)
=\gamma_7 - \sum_{\chi(-1) = -1} \frac{L'}{L}(1,\chi),
$$
a formula that, with $7$ replaced by any odd prime $q$, also holds in case $K=\mathbb Q(\zeta_q+\zeta_q^{-1})$
with $\gamma_q$ being the Euler-Kronecker constant of $\mathbb Q(\zeta_q)$, see, e.g., \cite[eqs.~(8) and (10)]{Msurvey}.
\end{Ex}

\section{A reinterpretation of a result of Serre \texorpdfstring{$(d=12)$}{d12}}
\label{sec:FouvrySerre}
In the introduction we mentioned an 
example by Serre \cite[pp.~185--187]{SerreChebotarev} and 
an analysis occurring in a 
paper of Fouvry, Levesque and Waldschmidt~\cite[pp.~80--84]{FLW} dealing with
$d=12$. Serre's example concerns Fourier coefficients of a certain eta-product and the work of Fouvry et al.\,simultaneous representation by 
two binary quadratic forms. We show that several of the results of the latter can be easily obtained from those of Serre.

Serre is interested in the lacunarity of the eta-product
$$f(z)=\eta(12z)^2=\sum_{n=1}^{\infty}a_nq^n=q\prod_{n=1}^{\infty}(1-q^{12n})^2.$$
We have $$f(z)=q-2q^{13}-q^{25}+2q^{37}+q^{49}+2q^{61}-2q^{97}-2q^{109}+q^{121}+2q^{157}+3q^{169}-2q^{181}+\ldots$$
Recall that a modular form $f$ is said to be \emph{lacunary} if the asymptotic density in $\mathbb N$ 
of the $n$ for which
$a_n(f)\ne 0$, is zero.
If $a$ and $d$ are two integers, we denote by $N_{a,d}$ any
positive integer that is either $1$ or composed of only primes $p\equiv a \pmod*{d}$.
Serre proves that $a_n$ is non-zero if and only if 
$n = (N_{5,12}N_{7,12} N_{11,12} )^2 N_{1,12}$.
Thus the associated set $\{n:a_n\ne 0\}$, call it ${\mathfrak M}$, 
equals $\{n:n=(N_{5,12}N_{7,12} N_{11,12} )^2 N_{1,12}\}$, which is multiplicative,
has $\delta=1/4$ and associated generating series 
$F(s)=\zeta_{12,5}(2s)\zeta_{12,7}(2s)\zeta_{12,11}(2s)\zeta_{12,1}(s)$.
By \eqref{starrie1} 
it then follows that, for arbitrary $j\ge 1$,
\begin{equation}
\label{starrie2}
\#\{\mathfrak M\cap [1,x]\}=\frac{x}{\log^{3/4}x}\Bigl(\beta_0+
\frac{\beta_1}{\log x}+\frac{\beta_2}{\log^2 x}+\dots+
\frac{\beta_j}{\log^j x}+O_j\Bigl(\frac{1}{\log^{j+1}x}\Bigr)\Bigr),
\end{equation}
with $\beta_1,\dotsc,\beta_j$ constants and $\beta_1$ satisfying $\beta_1=3\beta_0(1-\gamma_{\mathfrak M})/4$. Serre (Prop.\,19) shows that
\begin{equation}
\label{serrebeta}    
\beta_0=\frac{(2^{-1}3^{-7}\pi^6\log(2+\sqrt{3}))^{1/4}}{\Gamma(1/4)}\prod_{p\equiv 1\pmod*{12}}\sqrt{1-p^{-2}}=0.20154\ldots,
\end{equation}
where these days easily more decimals can be computed, giving 
$$\beta_0=0.2015440949047104252985132669867998104016\dotsc$$ 
Serre finds $\beta_0$ by relating
$F(s)$ to the Dedekind zeta function of $\mathbb Q(\zeta_{12})=\mathbb Q(i,\sqrt{3})$ and computing its residue.
It is very easy to go further and compute $\beta_1$ via $\gamma_{\mathfrak M}$, we have namely
$$\gamma_{\mathfrak M}=\gamma(12,1)-2\sum_{p\equiv 5,7,11\pmod*{12}}\frac{\log p}{p^2-1},$$
leading numerically to
$$\gamma_{\mathfrak M}=0.4218587490880596232477371177014182248836\dotsc$$ 
The asymptotic \eqref{starrie2} shows that $\eta(12z)^2$ is lacunary and quantifies the lacunarity\footnote{Serre \cite{serrelacunary} showed that $\eta^r$ with $r\ge 2$ an even integer is lacunary if and only if
$r\in \{2,4, 6, 8, 10, 14,26\}$.}.
Next we will discuss the connection of this result with \cite{FLW}.

It is classical that $n$ can be written as a sum of two
squares if and only if $n = 2^a N_{3,4}^2N_{1,4}$ for some 
$a\ge 0$. Likewise, $n$ is representable by $X^2+XY+Y^2$ if and only if $n = 3^b N_{2,3}^2N_{1,3}$ for some $b\ge 0$. It follows from this that $n$ is representable by both quadratic forms
if and only if $n = (2^a 3^b N_{5,12} N_{7,12} N_{11,12} )^2 N_{1,12}$.
The associated set $\mathfrak N$ is multiplicative and has $\delta=1/4$. By \eqref{starrie1} 
it then follows that, for arbitrary $j\ge 1$,
\begin{equation}
\label{starrie3}
\#\{\mathfrak N\cap [1,x]\}=\frac{x}{\log^{3/4}x}\Bigl(\beta'_0+
\frac{\beta'_1}{\log x}+\frac{\beta'_2}{\log^2 x}+\dots+
\frac{\beta'_j}{\log^j x}+O_j\Bigl(\frac{1}{\log^{j+1}x}\Bigr)\Bigr),
\end{equation}
with $\beta'_1,\dotsc,\beta'_j$ constants satisfying $\beta'_1=3\beta_0'(1-\gamma_{\mathfrak N})/4$.
We note that \eqref{starrie3} is \cite[(6.3)]{FLW}. Fouvry et al. 
continue their analysis by determining $\beta'_0$ using the generating series associated $G(s)$ to $\mathfrak N$,
which satisfies
$$G(s)=(1-4^{-s})^{-1}(1-9^{-s})^{-1}\zeta_{12,5}(2s)\zeta_{12,7}(2s)\zeta_{12,11}(2s)\zeta_{12,1}(s).$$
This can be cut short by comparing with Serre's work. Clearly 
$$G(s)=(1-4^{-s})^{-1}(1-9^{-s})^{-1}F(s),$$ giving rise to 
$\beta'_0=3\beta_0/2$ and 
$\gamma_{\mathfrak N}= \gamma_{\mathfrak M}- \frac23 \log2 - \frac14  \log 3$, leading
to the numerical values
$$\beta_0' = 0.3023161423570656379477699004801997156024\dotsc$$
and $$\gamma_{\mathfrak N} = 
- 0.3148924434522646725458956058346642466619\dotsc$$
Since $\zeta(2)=\pi^2/6$ we find
$$\prod_{p\equiv 5,7,11\pmod*{12}}\frac{1}{\sqrt{1-p^{-2}}}=\frac{\pi}{3}\prod_{p\equiv 1\pmod*{12}}\sqrt{1-p^{-2}}.$$
This in combination with \eqref{serrebeta} and $\beta'_0=3\beta_0/2$, then gives rise to the formula at the end of p.~83 of \cite{FLW}:
$$\beta'_0=\frac{3^{1/4}}{2^{5/4}}\sqrt{\pi}\frac{(\log(2+\sqrt{3}))^{1/4}}{\Gamma(1/4)}\prod_{p\equiv 5,7,11\pmod*{12}}\frac{1}{\sqrt{1-p^{-2}}}.$$

We leave it to the reader to
verify that Serre's criterion for $a_n$ to be non-zero can be reformulated in the following more elegant way.
\begin{Prop}
\label{surprisingidentity}
We have $a_n\ne 0$ if and only if $n\equiv 1 \pmod*{12}$
and $n$ is representable by both $X^2+Y^2$ and $X^2+XY+Y^2$.
\end{Prop}
Serre justifies his result by finding Hecke theta series for $\eta(12z)^2$ using Galois representations.
Huber et al.\,\cite{veel} reproved it using the theory of theta identities
(for the standard work relating eta products and theta identities, see
K\"ohler \cite{Kohl}), taking the identity
$$\eta(12z)^2=\frac12\sum_{n=1}^{\infty}\Big(\sum_{x^2+36y^2=n}1-
\sum_{4x^2+9y^2=n}1\Big)q^n.$$
as a starting point.
The elementary formulation of  Proposition \ref{surprisingidentity} suggests that perhaps 
an even more basic proof using a combinatorial approach \emph{\`a la} Ramanujan is 
possible!

\medskip
\noindent \textbf{Acknowledgment}. 
The second author thanks Heng Huat Chan, Henri Cohen, Bernhard Heim, Jens Funke, Lukas Mauth, Eric Mortenson, Robert Pollack, Jean-Pierre Serre, Tonghai Yang (who pointed out the article \cite{veel}), Don Zagier and 
Sander Zwegers for helpful feedback regarding 
Section \ref{sec:FouvrySerre} and the Max-Planck-Institute for Mathematics for making interactions with so many experts possible.

Part of the computations were performed on an assembled machine that uses an
Intel Core i7-13700KF, 5.4 GHz running Ubuntu 22.04.04 and with 64GB of RAM. 
Thanks are due to Luca Righi and Alessandro Talami (I.T. services of the University of Padova) 
for having assembled this machine. The second part of the computations was performed using 
240 GB of RAM on a 4 x Eight-Core Intel(R) Xeon(R) CPU E5-4640 @ 2.40GHz
machine of the cluster of the Dipartimento di Matematica of the University of Padova.

\newpage

\begin{table}[H]
\centering
\begin{tabular}{|c|c|r|r|}
\hline
$d$ & $a$ & $\gamma(d,a)$ \phantom{012345678901234567890}& \text{precision} \\
\hline
$1$ & $1$ &  $ 0.57721 56649 01532 86060 65120 90082 40243 10421\dotsc$ & $477\,511\,832\,674$ \\
\hline
$2$ & $1$ &  $ 1.27036 28454 61478 17002 37442 11540 57899 91176\dotsc$ & $477\,511\,832\,674$\\
\hline
$3$ & $1$ &  $ 1.09904 95258 66676 53048 44653 68305 61011 72687\dotsc$ & $130\,000$\\
    & $2$ &  $ 0.02747 22833 68911 17581 96693 40238 05516 60971\dotsc$ & $130\,000$\\
\hline
$4$ & $1$ &  $ 0.98672 25683 13428 62885 16284 85266 20006 80853\dotsc$ & $130\,000$\\
    & $3$ &  $ 0.28364 02771 48049 54117 21157 26274 37893 10323\dotsc$ & $130\,000$ \\
\hline
$5$ & $1$ &  $ 0.60852 39515 39242 47207 97932 58724 29778 44852\dotsc$ & $100$\\
    & $2$ &  $-0.35215 15687 27300 94966 16629 36463 23178 29793\dotsc$ & $100$\\
    & $3$ &  $-0.02772 97956 30966 82652 57700 61450 59717 05577\dotsc$ & $100$\\
    & $4$ &  $ 0.75093 25558 29083 25836 43416 62578 48050 99754\dotsc$ & $100$\\
\hline
$6$ & $1$ &  $ 1.09904 95258 66676 53048 44653 68305 61011 72687\dotsc$ & $130\,000$\\
    & $5$ &  $ 0.72061 94639 28856 48523 69014 61696 23173 41726\dotsc$ & $130\,000$\\
\hline
$7$ & $1$ &  $ 0.52481 50406 92497 35626 25223 41149 53328 99061\dotsc$ & $100$\\
    & $2$ &  $-0.20560 54381 56889 24645 34786 98566 56334 80791\dotsc$ & $100$\\
    & $3$ &  $-0.20737 92262 04487 01767 49633 27989 58199 80539\dotsc$ & $100$\\
    & $4$ &  $ 0.39413 62902 98642 49586 94917 73880 94933 98481\dotsc$ & $100$\\
    & $5$ &  $ 0.00220 39624 97149 58098 05463 31645 38511 14840\dotsc$ & $100$\\
    & $6$ &  $ 0.39336 33939 50505 24247 32857 93869 87665 75430\dotsc$ & $100$\\
\hline
$8$ & $1$ &  $ 0.85493 54647 72769 23354 33587 20452 45006 49277\dotsc$ & $100$\\
    & $3$ &  $-0.07982 91608 61031 25405 62783 91570 65001 34207\dotsc$ & $100$\\
    & $5$ &  $ 0.13178 71035 40659 39530 82697 64813 75000 31576\dotsc$ & $100$\\
    & $7$ &  $ 0.36346 94380 09080 79522 83941 17845 02894 44530\dotsc$ & $100$\\
\hline
$9$ & $1$ &  $ 0.45839 34116 75144 88707 34014 32170 12245 66066\dotsc$ & $100$\\
    & $2$ &  $-0.43909 16249 87433 55177 22607 23033 84463 56563\dotsc$ & $100$\\
    & $4$ &  $ 0.37664 35697 72260 74850 13387 88635 14239 81437\dotsc$ & $100$\\
    & $5$ &  $ 0.03103 77589 75523 72540 80223 10280 02848 01451\dotsc$ & $100$\\
    & $7$ &  $ 0.26401 25444 19270 89490 97251 47500 34526 25184\dotsc$ & $100$\\
    & $8$ &  $ 0.43552 61493 80821 00218 39077 52991 87132 16083\dotsc$ & $100$\\
\hline
$12$ & $1$ & $ 0.78044 52516 80074 82217 81507 48784 08502 59477\dotsc$ & $100$\\
    & $5$ &  $ 0.20627 73166 33353 80667 34777 36482 11504 21376\dotsc$ & $100$\\
    & $7$ &  $ 0.31860 42741 86601 70830 63146 19521 52509 13210\dotsc$ & $100$\\
    & $11$ & $ 0.51434 21472 95502 67856 34237 25214 11669 20349\dotsc$ & $100$\\
\hline
\end{tabular}
\caption{{Truncated values ($40$ decimals) of some Euler constants 
from primes in arithmetic progressions. The final column records
the largest number of decimals known for the corresponding value of $\gamma(d,a)$.
The accuracy for $\gamma(1,1)=\gamma$ and 
$\gamma(2,1) =\gamma + \log 2$ comes from the results on the
\urltwo{http://www.numberworld.org/digits/}{numberworld\_website}.
Moreover, $\gamma(3,2) = -\gamma(3,1) + \gamma + (\log 3)/2$, $\gamma(6,1)=\gamma(3,1)$ and
$\gamma(6,5) = \gamma(3,2) + \log 2$. 
For $d=3,4$ the decimals given agree with those 
determined by Moree \cite{bias} (33, respectively 26 decimals).
}}
\label{tab:24}
\end{table}

\bigskip
\begin{flushleft} 
Alessandro Languasco, Universit\`a di Padova,\\
Dipartimento di Matematica ``Tullio Levi-Civita'', \\
via Trieste 63, 35121 Padova, Italy.\\
e-mail: alessandro.languasco@unipd.it
\end{flushleft}

\bigskip
\begin{flushleft}
Pieter Moree, Max Planck Institute for Mathematics\\ 
Vivatsgasse 7, 53111, Bonn, 
Germany.\\  
e-mail: moree@mpim-bonn.mpg.de 
\end{flushleft}

\newpage
\appendix

\section{The \texorpdfstring{$\gamma(d,a)$}{gdaformulae} formulae for \texorpdfstring{$d=5,7,8,9,12$}{variousd}}
\label{sec:appendix}
We record here how the formulae of Theorem \ref{mainabelian} 
can be written in a more explicit way for some small values of $d$.
For this we use \eqref{compacttwo}.
These formulae can also be used
to double-check the numerical results obtained with the algorithm
presented in Section \ref{sec:algorithm}, see also Remark \ref{Remark-sums}.
We recall that $\gamma_1(d,a)$ is defined in \eqref{gamma1-def}.

\subsection{The case \texorpdfstring{$d=5$}{d5}.}
We have
\begin{align*}
\gamma(5,1)
&=
\gamma_1(5,1)
+ 
\sum_{p\equiv 2,3 \pmod* 5}  \frac{\log p}{p^{4}-1}
+
\sum_{p\equiv 4 \pmod* 5}  \frac{\log p}{p^{2}-1};
\\
\gamma(5,2)
&=
\gamma_1(5,2)
+  
\sum_{p\equiv 3 \pmod* 5}\! \frac{p\,\log p}{p^{4}-1}
- 
\sum_{p\equiv 2 \pmod* 5}\!\!\!\!\! \log p\, \frac{p^2+p+1}{p^{4}-1};
\\
\gamma(5,3)
&=
\gamma_1(5,3)
+ 
\sum_{p\equiv 2 \pmod* 5} \!\!\ \frac{p\,\log p}{p^{4}-1}
-
\sum_{p\equiv 3 \pmod* 5} \!\!\!\!\! \log p\, \frac{p^2+p+1}{p^{4}-1};
\\
\gamma(5,4)
&=
\gamma_1(5,4)
+ 
\sum_{p\equiv 2,3 \pmod* 5} \frac{p^2\,\log p}{p^{4}-1}
-
\sum_{p\equiv 4 \pmod* 5} \frac{\log p}{p^{2}-1}.
\end{align*}

\subsection{The case \texorpdfstring{$d=7$}{d7}.}
\label{sec:d7} We have
\begin{align*}
\gamma(7,1)
&=
\gamma_1(7,1)
+  
\sum_{p\equiv 2,4 \pmod* 7}  \frac{\log p}{p^{3}-1}
+ 
\sum_{p\equiv 3,5 \pmod* 7}  \frac{\log p}{p^{6}-1}
+ 
\sum_{p\equiv 6 \pmod* 7}  \frac{\log p}{p^{2}-1};
\\
\gamma(7,2)
&=
\gamma_1(7,2)
+  
\sum_{p\equiv 3 \pmod* 7}\frac{p^4\, \log p}{p^{6}-1}
+  
\sum_{p\equiv 4 \pmod* 7} \frac{p\, \log p}{p^{3}-1}
+  
\sum_{p\equiv 5 \pmod* 7} \frac{p^2\, \log p}{p^{6}-1}
\\&\hskip1.8cm
- 
\sum_{p\equiv 2 \pmod* 7} \frac{(p+1)\log p}{p^{3}-1};
\\
\gamma(7,3)
&=
\gamma_1(7,3)
+  
\sum_{p\equiv 5 \pmod* 7}\frac{p\, \log p}{p^{6}-1}
- 
\sum_{p\equiv 3 \pmod* 7}\!\!\!\!\! \log p\, \frac{p^4+p^3+p^2+p+1}{p^{6}-1};
\\
\gamma(7,4)
&=
\gamma_1(7,4)
+  
\sum_{p\equiv 2 \pmod* 7} \frac{p\, \log p}{p^{3}-1}
+  
\sum_{p\equiv 3 \pmod* 7} \frac{p^2\, \log p}{p^{6}-1}
+  
\sum_{p\equiv 5 \pmod* 7} \frac{p^4\, \log p}{p^{6}-1}
\\&\hskip1.8cm
- 
\sum_{p\equiv 4 \pmod* 7} \frac{(p+1)\log p}{p^{3}-1};
\\
\gamma(7,5)
&=
\gamma_1(7,5)
+  
\sum_{p\equiv 3 \pmod* 7} \frac{p\, \log p}{p^{6}-1}
- 
\sum_{p\equiv 5 \pmod* 7}\!\!\!\!\! \log p\, \frac{p^4+p^3+p^2+p+1}{p^{6}-1};
\\
\gamma(7,6)
&=
\gamma_1(7,6)
+  
\sum_{p\equiv 3,5 \pmod* 7}\!\!\frac{p^3\,\log p}{p^{6}-1} 
- 
\sum_{p\equiv 6 \pmod* 7} \frac{\log p}{p^{2}-1}.
\end{align*}

\subsection{The case \texorpdfstring{$d=8$}{d8}.}
We have
\begin{align*}
\gamma(8,1)
&=
\gamma_1(8,1)
+ 
\sum_{p\equiv 3,5,7 \pmod*{8}} \frac{\log p}{p^2-1};
\\
\gamma(8,j)
&=
\gamma_1(8,j)
-
\sum_{p\equiv j \pmod*{8}} \frac{\log p}{p^2-1} \quad \text{~for~}j=3,5,7.
\end{align*}

\subsection{The case \texorpdfstring{$d=9$}{d9}.}
We have
\begin{align*}
\gamma(9,1)
&=
\gamma_1(9,1)
+ 
\sum_{p\equiv 2,5 \pmod* 9}  \frac{\log p}{p^{6}-1}
+
\sum_{p\equiv 4,7 \pmod* 9}  \frac{\log p}{p^{3}-1}
+
\sum_{p\equiv 8 \pmod* 9}  \frac{\log p}{p^{2}-1};
\\
\gamma(9,2)
&=
\gamma_1(9,2)
+  
\sum_{p\equiv 5 \pmod* 9}\frac{p\, \log p}{p^{6}-1}
- 
\sum_{p\equiv 2 \pmod* 9}\!\!\!\!\! \log p\, \frac{p^4+p^3+p^2+p+1}{p^{6}-1};
\\
\gamma(9,4)
&=
\gamma_1(9,4)
+  
\sum_{p\equiv 2 \pmod* 9} \frac{p^4\, \log p}{p^{6}-1}
+  
\sum_{p\equiv 5 \pmod* 9}\frac{p^2\, \log p}{p^{6}-1}
+  
\sum_{p\equiv 7 \pmod* 9} \frac{p\, \log p}{p^{3}-1}
\\&\hskip1.8cm
- 
\sum_{p\equiv 4 \pmod* 9}\frac{(p+1)\log p}{p^{3}-1};
\\
\gamma(9,5)
&=
\gamma_1(9,5)
+  
\sum_{p\equiv 2 \pmod* 9}\frac{p\, \log p}{p^{6}-1}
- 
\sum_{p\equiv 5 \pmod* 9} \log p\, \frac{p^4+p^3+p^2+p+1}{p^{6}-1};
\\
\gamma(9,7)
&=
\gamma_1(9,7)
+  
\sum_{p\equiv 2 \pmod* 9}\frac{p^2\, \log p}{p^{6}-1} 
+  
\sum_{p\equiv 4 \pmod* 9}\!\frac{p\,\log p}{p^{3}-1}
+  
\sum_{p\equiv 5 \pmod* 9}\!\frac{p^4\,\log p}{p^{6}-1} 
\\& \hskip1.8cm
- 
\sum_{p\equiv 7 \pmod* 9}\frac{(p+1)\log p}{p^{3}-1};
\\
\gamma(9,8)&
=
\gamma_1(9,8)
+  
\sum_{p\equiv 2,5 \pmod* 9}\frac{p^3\, \log p}{p^{6}-1}  
- 
\sum_{p\equiv 8 \pmod* 9}  \frac{\log p}{p^{2}-1}.
\end{align*}

\subsection{The case \texorpdfstring{$d=12$}{d12}.}
We have
\begin{align*}
\gamma(12,1)
&=
\gamma_1(12,1)
+  
\sum_{p\equiv 5,7,11 \pmod*{12}} \frac{\log p}{p^2-1};
\\
\gamma(12,j)
&=
\gamma_1(12,j)
-  
\sum_{p\equiv j \pmod*{12}} \frac{\log p}{p^2-1}\quad \text{~for~}j=5,7,11.
\end{align*}

\end{document}